\documentclass[a4paper,10pt]{article}

\usepackage[utf8]{inputenc}
\usepackage[english]{babel}

\usepackage{mathtools}
\usepackage{amsmath}
\usepackage{amssymb}
\usepackage{dsfont}
\usepackage{amsthm}
\usepackage{cases}
\usepackage{stix}
\usepackage{pgf,tikz}
\usetikzlibrary{arrows}\usepackage[pdftex,colorlinks]{hyperref}
\hypersetup{                    %
	      allcolors = blue}
	      
\usepackage[textwidth=6in,textheight=9.25in, paperwidth=7.5in,paperheight=11.25in, inner=2.5cm, outer=2.5cm]{geometry} 

\renewcommand{\bar}[1]{\overline{#1}}

\theoremstyle{definition} \newtheorem{defn}{Definition}[section]
\theoremstyle{remark} 
\theoremstyle{plain} \newtheorem{thm}[defn]{Theorem}
		     \newtheorem*{theorem*}{Theorem}
\theoremstyle{plain} \newtheorem{cor}[defn]{Corollary}
		     
\theoremstyle{plain} \newtheorem{prop}[defn]{Proposition}
\theoremstyle{plain} \newtheorem{lem}[defn]{Lemma}
\theoremstyle{remark} \newtheorem{rem}[defn]{Remark}
\theoremstyle{remark} 
		      \newtheorem*{rmk*}{Remark}
\theoremstyle{plain} \newtheorem*{princ*}{Principle}

\usepackage{enumitem}

\usepackage{pgf,tikz}
\usepackage{mathrsfs}
\usetikzlibrary{arrows}
\usetikzlibrary[patterns]
\definecolor{ffqqqq}{rgb}{1,0,0}
\definecolor{ududff}{rgb}{0.30196078431372547,0.30196078431372547,1}

\usepackage[affil-it]{authblk}

\author{Mark Wilkinson\thanks{Department of Mathematics and Computer Science (MACS), Heriot-Watt University, Edinburgh, EH14 4AS. }}
\title{Local-in-time Physical Solutions of the Incompressible Semi-Geostrophic Equations in Eulerian Coordinates}
\date{}

\hypersetup{citecolor=red}

\usepackage{mathrsfs}

\usepackage{amsmath}

\begin{document}

\maketitle

\begin{abstract}
\noindent We prove the existence of local-in-time smooth solutions of the incompressible semi-geostrophic equations expressed in Eulerian co-ordinates in 3-dimensional smooth bounded simply-connected domains. Our solutions adhere to Cullen's Stability Principle in that the geopotential is guaranteed to be a convex map for all times of its existence. We achieve our results by appealing to the theory of so-called $\mathrm{div}$-$\mathrm{curl}$ systems (or {\em Hodge} systems), making use of recent results of Wang, which yield useful estimates on the ageostrophic velocity field. To our knowledge, this work constitutes the first time that any notion of bounded solution of the semi-geostrophic equations in Eulerian co-ordinates has been constructed on a bounded domain. Indeed, our work solves an open problem as highlighted by, among others, A. Figalli in his CIME lectures on the semi-geostrophic equations. Our methods are largely elementary. We discuss the application of the novel ideas in this work to the case of {\em variable} Coriolis force in the final section of the article.
\end{abstract}

\section{Introduction}
In this article, we study the incompressible semi-geostrophic equations on a bounded, smooth, open, simply-connected domain $\Omega$, which are given by
\begin{equation}\label{SG}
\left\{
\begin{array}{l}
\displaystyle \frac{\partial u_{g}}{\partial t}+(u\cdot\nabla)u_{g}=-Ju_{a}, \vspace{2mm}\\
\displaystyle \frac{\partial\theta}{\partial t}+(u\cdot\nabla)\theta=0, \vspace{2mm} \\
\nabla\cdot u=0,
\end{array}
\right.\tag{SG}
\end{equation}
where $u_{g}:\Omega\times [0, \infty)\rightarrow\mathbb{R}^{3}$ denotes the {\em geostrophic velocity field}
\begin{equation*}
u_{g}(x, t)=\left(
\begin{array}{c}
u_{g, 1}(x, t) \\
u_{g, 2}(x, t) \\
0
\end{array}
\right),
\end{equation*}
$u_{a}:\Omega\times [0, \infty)\rightarrow\mathbb{R}^{3}$ denotes the {\em ageostrophic velocity field} defined by
\begin{equation*}
u_{a}:=u-u_{g}
\end{equation*}
and $\theta: \Omega\times [0, \infty)\rightarrow\mathbb{R}$ denotes the so-called {\em buoyancy anomaly}. Finally, $J\in\mathbb{R}^{3\times 3}$ denotes the matrix
\begin{equation*}
J:=\left(
\begin{array}{ccc}
0 & -1 & 0 \\
1 & 0 & 0 \\
0 & 0 & 0
\end{array}
\right).
\end{equation*} In this work, the Eulerian velocity field $u$ is also subject to the boundary condition 
\begin{equation*}
u\cdot n=0\quad \text{on}\hspace{2mm}\partial\Omega,
\end{equation*}
where $n:\partial\Omega\rightarrow\mathbb{S}^{2}$ denotes the outwardly-directed normal field to the boundary of the domain $\partial \Omega$. The unknown fields $u_{g}$ and $\theta$ are not independent, but are rather linked through the so-called {\em geopotential} $\phi:\Omega\times [0, \infty)\rightarrow \mathbb{R}$ by the relation
\begin{equation*}
\nabla \phi(x, t)=\left(
\begin{array}{c}
u_{g, 2}(x, t) \\
-u_{g, 1}(x, t) \\
\theta(x, t)
\end{array}
\right).
\end{equation*}
This system, first introduced by Eliasson \cite{eliassen} in 1949 and later rediscovered by Hoskins \cite{hoskins1975geostrophic} in 1975, is believed by many to be a reasonable model for the development, and subsequent dynamics, of atmospheric fronts. We invite the reader to consult the book of Cullen \cite{cullen2006mathematical} for more information on the physics of the model. The mathematical structure of system \eqref{SG} is perhaps not immediately apparent in the current form in which it appears, notably due to the absence of a law of evolution for the Eulerian velocity field $u$. However, following the approach of Benamou and Brenier \cite{benamou1998weak}, if one introduces the {\em generalised geopotential} $P$ defined pointwise on $\Omega\times [0, \infty)$ by 
\begin{equation*}
P(x, t):=\phi(x, t)+\frac{1}{2}(x_{1}^{2}+x_{2}^{2}),
\end{equation*}
one finds that the system \eqref{SG} takes the form of the following semi-linear transport equation
\begin{equation}\label{AT}
\frac{\partial T}{\partial t}+(u\cdot\nabla)T=J(T-\mathrm{id}_{\Omega}),\tag{AT}
\end{equation}
subject to the non-linear constraint that $T$ be a time-dependent conservative vector field on $\Omega$, i.e. 
\begin{equation*}
T=\nabla P.
\end{equation*}
Moreover, $\mathrm{id}_{\Omega}(x):=x$ for all $x\in\Omega$. We refer to this equation as (AT) since it is an example of an {\em active transport} equation in an unknown (conservative) vector field. Indeed, the full Eulerian velocity field $u$ is neither a given datum, nor is it required to satisfy a prescribed evolution equation, but is rather determined only by the condition that it advect $T$ in such a way that it remain conservative on $\Omega$ for all times. In their work \cite{benamou1998weak}, Benamou and Brenier proceed to re-express equation \eqref{AT} in time-dependent {\em geostrophic coordinates} $X_{t}:\Omega\rightarrow X_{t}(\Omega)$ given pointwise by
\begin{equation}
X_{t}(x):=\nabla P(x, t)\quad \text{for}\hspace{2mm}x\in\Omega
\end{equation}
for each time $t$. In this coordinate system, \eqref{AT} takes the form of a scalar transport equation coupled with a Monge-Amp\`{e}re equation. In contrast, we find a method by which to analyse system \eqref{AT} directly in natural Eulerian coordinates, rather than passing to Lagrangian or geostrophic coordinates.

This paper tackles and solves the open problem pertaining to the construction of local-in-time classical solutions of \eqref{SG} in Eulerian co-ordinates on bounded simply-connected domains whose associated geostrophic measure is of bounded support in $\mathbb{R}^{3}$. Indeed, whilst this problem has been of concern since the original pioneering work of Benamou and Brenier \cite{benamou1998weak}, it was raised once again relatively recently by Figalli in his 2014 CIME lectures \cite{figs} (we invite the reader to consult open problem 3 therein). As far as a local-in-time theory of the incompressible semi-geostrophic equations is concerned, our work can be considered as the culmination of a sequence of efforts, including the work of Benamou and Brenier \cite{benamou1998weak}, Loeper \cite{loeper2006fully}, Feldman and Tudorascu \cite{feldman2015semi}, Ambrosio, Colombo, De Philippis and Figalli \cite{ambrosio2014global}, and Cheng, Cullen and Feldman \cite{Cheng2018}, among others. We obtain our results by exploiting certain div-curl structure in the semi-geostrophic equations,	and we believe this to be the first work on \eqref{SG} which does so.
%
%
%
\subsection{The Problem of Boundedness of Support of the Geostrophic Measure}
An important quantity in the analysis of the semi-geostrophic equations is the so-called \textbf{geostrophic measure} defined by
\begin{equation}\label{geomeasure}
\nu_{t}:=\nabla P(\cdot, t)\#\mathscr{L}_{\Omega}\quad \text{on}\hspace{2mm}\mathbb{R}^{3},
\end{equation} 
where $\mathscr{L}_{\Omega}$ denotes the normalised restriction of the Lebesgue measure on $\mathbb{R}^{3}$ to $\Omega$. We would like to emphasise that the solutions we construct in this work admit the physically-desirable property that $\nabla P(\cdot, t)\in L^{\infty}$ for all times $t$ of existence. Prior to this work, the best result with regards to solutions of the system \eqref{AT} in bounded domains was achieved by Ambrosio, Colombo, De Philippis and Figalli \cite{ambrosio2014global}, in which the authors were able to construct global-in-time distributional solutions of \eqref{AT} but {\em only} in the case that the fluid domain $\Omega\subset\mathbb{R}^{3}$ is convex. Perhaps even more unsatisfyingly, owing to the fact that the geostrophic measure $\nu_{t}$ corresponding to their solutions must always be strictly globally-supported on $\mathbb{R}^{3}$, their result also suffers from the physical defect that the temperature field $\theta(\cdot, t)$ cannot lie in $L^{\infty}(\Omega)$ for {\em any} time $t$. 

In contrast to the work in \cite{ambrosio2014global}, our domain $\Omega$ need only be open, bounded and simply connected with suitably-smooth boundary $\partial\Omega$. More importantly, the geostrophic measure $\nu_{t}$ built through our smooth solutions $(\nabla P, u)$ is of bounded support in $\mathbb{R}^{3}$ (considered as geostrophic space) for all times $t$. The reader will note that if the support of the measure \eqref{geomeasure} is the whole space $\mathbb{R}^{3}$, then the temperature field $\theta(\cdot, t)$ cannot lie in $L^{\infty}(\Omega)$ for {\em any} time $t$. It is for this reason we term solutions $(\nabla P, u)$ of \eqref{AT} with the property that the associated measure \eqref{geomeasure} is always of bounded support to be \textbf{physical} in this article.

Surprisingly, we achieve our results by rather elementary means, making use of new estimates on inhomogeneous div-curl systems due to Wang \cite{wang2018solvability}. The aforementioned results allow us to derive -- previously inaccessible -- a priori estimates on the full Eulerian velocity field $u$.
%
%
\subsection{Cullen's Stability Principle}
Throughout this article, we also employ the idea due to Cullen (see Cullen \cite{cullen2006mathematical}, chapter 3) that a solution $(\nabla P, u)$ of \eqref{SG} is to be considered as \textbf{stable} if and only if the geopotential $\phi(\cdot, t)$ is a convex function on $\Omega$ for all times $t$. This convexity condition, whilst convenient from the point of view that it ensures \eqref{SG} admit ellipticity properties, can be formulated by means of energetics. The reader may verify that the semi-geostrophic system admits the energy functional
\begin{equation}\label{geoenergy}
E_{\Omega}[T]:=\frac{1}{2}\int_{\Omega}\left((x_{1}-T_{1}(x))^{2}+(x_{2}-T_{2}(x))^{2}-2x_{3}T_{3}(x)\right)\,dx,
\end{equation}
whose value is formally constant along solution trajectories. It can be shown, either by using the notion of inner variation (see Giaquinta and Hildebrandt \cite{giaquinta2004calculus}, section 3.3) or by using the theory optimal transport, that a solution $(\nabla P, u)$ is stable if and only if $\nabla P(\cdot, t)$ minimises \eqref{geoenergy} in `suitable' admissible classes. We expound no longer on Cullen's Stability Principle in this article, but simply choose in the sequel to solve \eqref{AT} in a class of conservative vector fields which derive from a convex potential.
\subsection{Statement of Main Results}
Let us now set out the notion of solution of the active transport equation with which we shall work throughout the sequel. In all the following, we work with the class of uniformly convex initial geopotential fields $\mathcal{P}$ given by
\begin{equation*}
\mathcal{P}:=\left\{
\nabla P_{0}\in W^{3, p}(\Omega, \mathbb{R}^{3})\,:\,D^{2}P_{0}(x)\xi\cdot\xi\geq \lambda|\xi|^{2}\hspace{2mm}\text{for some}\hspace{1mm}\lambda=\lambda(P_{0})>0\hspace{2mm}\text{and all}\hspace{1mm}\xi\in\mathbb{R}^{3}
\right\},
\end{equation*}
where it is assumed that the Lebesgue exponent $p$ is strictly greater than 3.
\begin{defn}[Local-in-time Physical Classical Solutions of \eqref{AT}]
Suppose $\nabla P_{0}\in \mathcal{P}$ and $\alpha\in (0, 1)$. We say that the maps $\nabla P$ and $u$ constitute an associated \textbf{local-in-time physical classical solution} of \eqref{AT} if and only if 
\begin{equation*}
\nabla P\in C^{1}(0, \tau; C^{2, \alpha}(\bar{\Omega}, \mathbb{R}^{3}))\quad \text{and}\quad u\in C^{0}(0, \tau; C^{0, \alpha}_{\sigma}(\bar{\Omega}, \mathbb{R}^{3}))
\end{equation*}
for some $\tau>0$, and satisfy
\begin{equation*}
\frac{\partial}{\partial t}\nabla P(x, t)+D^{2}P(x, t)u(x, t)=J(\nabla P(x, t)-x)
\end{equation*}
pointwise in the classical sense for all $x\in\Omega$ and $t\in (0, \tau_{\ast})$. Moreover, $\nabla P$ admits the uniform convexity property
\begin{equation*}
D^{2}P(x, t)\xi\cdot \xi\geq \lambda_{\ast}|\xi|^{2}\quad \text{for all}\hspace{2mm}\xi\in\mathbb{R}^{3}
\end{equation*}
for some $\lambda_{\ast}>0$, all $x\in\Omega$ and $t\in [0, \tau)$. Finally, $\nabla P(\cdot, 0)=\nabla P_{0}$.
\end{defn}
The following is the main result of this work.
\begin{thm}\label{mainres}
For any $\nabla P_{0}\in \mathcal{P}$, there exists an associated local-in-time physical classical solution of \eqref{AT}.
\end{thm}
We also have the following important physical corollary of our main result, which settles the open problem as posed by Figalli (see \cite{figs}, open problem 3) in the setting of local-in-time classical solutions of \eqref{AT}.
\begin{cor}
For any $\nabla P_{0}\in\mathcal{P}$, the associated local-in-time physical classical solution $\nabla P$ of \eqref{AT} admits the property that
\begin{equation*}
\mathrm{supp}\,\nabla P(\cdot, t)\#\mathscr{L}_{\Omega}\subset\subset\mathbb{R}^{3}
\end{equation*}
for all $t\in [0, \tau_{\ast}]$.
\end{cor}
We claim that the methods outlined in this paper also extend, in a straightforward manner, to the case of the incompressible semi-geostrophic equations subject to a variable Coriolis force. We discuss this in the final section \ref{final} below.
\begin{rem}
The reader will note that no claim on uniqueness of our local-in-time physical classical solutions is present in the statement of our main result. However, given that even the {\em existence} of local-in-time physical classical solutions of \eqref{AT} on bounded domains was not even known prior to this work, we contend this article constitutes an important step forward in the analysis of the incompressible semi-geostrophic equations in Eulerian coordinates. The difficulty in demonstrating uniqueness of smooth solutions on a bounded domain is due to the fact that the known techniques for doing so are amenable only to the case in which $\nabla P(\cdot, t)\#\mathscr{L}_{\Omega}\ll \mathscr{L}_{\mathbb{R}^{3}}$ and for whose associated density $\beta(\cdot, t)$ is positive, bounded, and bounded strictly away from 0. We refer the reader to the work of Feldman and Tudorascu \cite{feldman2017semi} for more on this topic. 
\end{rem}
\subsection{Notation}
For any integer $m\geq 0$ and real number $1\leq p\leq \infty$, we write $W^{m, p}(\Omega, \mathbb{R}^{3})$ to denote the Sobolev space of $\mathbb{R}^{3}$-valued maps on $\Omega$, while $W^{m, p}_{\sigma}(\Omega, \mathbb{R}^{3})$ denotes the set of all maps whose distributional divergence vanishes on $\Omega$. We also write $C^{k, \alpha}(\bar{\Omega}, \mathbb{R}^{3})$ and $C^{k, \alpha}_{\sigma}(\bar{\Omega}, \mathbb{R}^{3})$ to denote the analogous H\"{o}lder spaces. If $(X, \|\cdot\|_{X})$ is a given Banach space and $\tau>0$, we write $\mathrm{BV}([0, \tau]; X)$ to denote the class of all maps of bounded variation on the time interval $[0, \tau]\subset\mathbb{R}$ with range in $X$, and we write $\mathrm{Var}(T; [0, \tau])$ to denote the variation of a map $T$ therein given by
\begin{equation*}
\mathrm{Var}(T; [0, \tau]):=\sup\left\{
\sum_{k=0}^{K(\mathsf{P})-1}\|T(t_{k+1})-T(t_{k})\|_{X}\,:\,\mathsf{P}=\{t_{k}\}_{k=0}^{K(\mathsf{P})-1}\hspace{2mm} \text{is a partition of}\hspace{1mm}[0, \tau]
\right\}.
\end{equation*}
We write $\mathds{1}_{I}:\mathbb{R}\rightarrow \{0, 1\}$ to denote the characteristic function of a subset $I\subset\mathbb{R}$ of the real line. Finally, we say a map $P\in C^{2}(\bar{\Omega})$ is $\lambda$-uniformly convex on $\Omega$ for some $\lambda>0$ if and only if
\begin{equation*}
D^{2}P(x)\xi\cdot\xi\geq \lambda|\xi|^{2}\quad \text{for all}\hspace{2mm}\xi\in\mathbb{R}^{3}
\end{equation*} 
and all $x\in\Omega$.
\section{A Priori Estimates for (SG)}\label{apriori}
We now derive simple and useful a priori estimates on the geopotential $P$. Moreover, utilising the work of Wang \cite{wang2018solvability}, establish {\em new} a priori estimates on the Eulerian velocity field $u$ for \eqref{AT} (and thereby also the {\em ageostrophic} velocity field $u_{a}$) which, to our knowledge, have never heretofore been applied in the analysis of the semi-geostrophic equations.
\subsection{Estimates on the Geopotential $P$}
We seek useful a priori estimates on both the gradient of the generalised geopotential $\nabla P$ and on the Eulerian velocity field $u$. In the pursuit of such estimates, we begin by supposing that for a given $P_{0}\in \mathcal{P}$ there exist $P$ and $u$ with
\begin{equation*}
P\in C^{1}((0, \tau_{\ast}); C^{3, \alpha}(\Omega)) \quad \text{and}\quad u\in C^{0}([0, \tau_{\ast}]; C^{1, \alpha}_{\sigma}(\Omega, \mathbb{R}^{3}))
\end{equation*}
which satisfy \eqref{AT} pointwise in the classical sense on $\Omega\times (0, \tau_{\ast})$ for some $\tau_{\ast}>0$. Moreover, we assume the existence of a compact set $K\subset\subset\mathbb{R}^{3\times 3}_{+}$ such that
\begin{equation}\label{goodset}
D^{2}P(x, t)\in K\quad \text{for all}\hspace{2mm} (x, t)\in\Omega\times [0, \tau_{\ast}).
\end{equation}
Now, under these assumptions, by multiplying throughout both sides of \eqref{AT} by
\begin{equation*}
|\nabla P(x, t)|^{p-2}\nabla P(x, t)
\end{equation*}
and using the fact $u(\cdot, t)$ is incompressible on $\Omega$ for each time $t$ -- together with the flux constraint that $u\cdot n=0$ on the boundary $\partial\Omega$ for all times -- one finds that
\begin{equation*}
\frac{1}{p}\frac{d}{dt}\int_{\Omega}|\nabla P(x, t)|^{p}\,dx\leq \int_{\Omega}|\nabla P(x, t)|^{p}\,dx+\max_{y\in\bar{\Omega}}|y|\int_{\Omega}|\nabla P(x, t)|^{p-1}\,dx,
\end{equation*}
from which it follows by H\"{o}lder's inequality and Gr\"{o}nwall's lemma that
\begin{equation}\label{grad}
\|\nabla P(\cdot, t)\|_{L^{p}(\Omega)}\leq \|\nabla P_{0}\|_{L^{p}(\Omega)}e^{(m+1)\tau_{\ast}}+C
\end{equation}
for some constant $C=C(\Omega, p)>0$ independent of $P_{0}$, where $m:=\max_{y\in\bar{\Omega}}|y|$. Finally, owing to the assumption that \eqref{goodset} hold true, we may conclude that
\begin{equation}\label{apriorip}
\left(\int_{0}^{\tau_{\ast}}\|\nabla P(\cdot, t)\|^{q}_{W^{1, p}(\Omega, \mathbb{R}^{3})}\,dt\right)^{\frac{1}{q}}\leq C
\end{equation}
for all $1\leq q\leq \infty$, for some constant $C=C(p, q, \nabla P_{0})$. With this space-time estimate in place, we advance to the consideration of estimates on the velocity field $u$ in \eqref{AT}.
\subsection{Estimates on the Eulerian Velocity Field $u$}
In order to find estimates for the full velocity field $u$, we appeal to estimates on solutions of so-called $\mathrm{div}$-$\mathrm{curl}$ systems. These are vector PDE systems of the shape
\begin{equation}\label{divcurl}
\left\{
\begin{array}{l}
\nabla\wedge(Au)=F \quad \text{on}\hspace{2mm}\Omega, \vspace{2mm}\\
\nabla\cdot u=G\quad \text{on}\hspace{2mm}\Omega, \vspace{2mm}\\
u\cdot n=H \quad \text{on}\hspace{2mm}\partial\Omega
\end{array}
\right.
\end{equation}
in the unknown map $u:\Omega\rightarrow\mathbb{R}^{3}$, where $A:\Omega\rightarrow\mathbb{R}^{3\times 3}, F:\Omega\rightarrow\mathbb{R}^{3}, G:\Omega\rightarrow\mathbb{R}$ and $H:\partial\Omega\rightarrow\mathbb{R}$ are given data. Systems of this type have been studied by Bourgain and Br\'{e}zis \cite{bourgain2004new}, Cheng and Shkoller \cite{cheng2017solvability}, and Wang \cite{wang2018solvability}, among others. To appreciate the relevance of systems of type \eqref{divcurl} in the analysis of the semi-geostrophic equations, we take curls throughout the equation \eqref{AT} to find that the Eulerian velocity field $u$ satisfies the (elliptic) system
\begin{equation*}
\nabla\wedge(D^{2}P(x, t)u(x, t))=\nabla\wedge J(\nabla P(x, t)-x)
\end{equation*}
for $x\in \Omega$. Manifestly, it follows that for each time $t\in [0, \tau_{\ast}]$ the velocity field $u(\cdot, t)$ satisfies system \eqref{divcurl} pointwise in the classical sense on $\Omega$, when the data are chosen such that
\begin{align*}
A(x):=D^{2}P(x, t), \vspace{2mm}\\
F(x):=\nabla\wedge J(\nabla P(x, t)-x), \vspace{2mm}\\
G(x):=0, \vspace{2mm} \\
H(x):=0,
\end{align*}
treating time $t$ as a parameter. It is at this point we appeal to recent results in \cite{wang2018solvability} on $W^{m, p}(\Omega, \mathbb{R}^{3})$-estimates of solutions of the $\mathrm{div}$-$\mathrm{curl}$ system above. We quote the following result from Wang for the special case that both $G$ and $H$ are the zero map in their respective classes. 
\begin{prop}[\cite{wang2018solvability}, Theorem 2.1]
Suppose $\Omega\subset\mathbb{R}^{3}$ is a bounded, open, simply connected set with smooth boundary $\partial\Omega$. Let $m\geq 1$, $1<p<\infty$ and $\alpha\in (0, 1)$ be given. Suppose $A\in C^{1, \alpha}(\bar{\Omega}, \mathbb{R}^{3\times 3})$ satisfies the uniform convexity condition
\begin{equation*}
A(x)\xi\cdot\xi\geq \lambda |\xi|^{2}\quad\text{for all}\hspace{2mm}\xi\in\mathbb{R}^{3}
\end{equation*}
for some $\lambda>0$ which is independent of $x\in\Omega$, and that $F\in W^{m-1, p}(\Omega)$ obeys the compatibility conditions
\begin{equation*}
\nabla\cdot F=0\quad \text{in the sense of distributions on}\hspace{1mm}\Omega
\end{equation*}
and
\begin{equation*}
\langle F\cdot n, \mathds{1}_{\partial\Omega}\rangle_{\partial\Omega}=0.
\end{equation*}
There exists a unique map $u\in W^{m, p}(\Omega, \mathbb{R}^{3})$ which satisfies the system
\begin{equation*}
\left\{
\begin{array}{l}
\nabla\wedge (Au)=F, \vspace{2mm}\\
\nabla\cdot u=0, \vspace{2mm}\\
u\cdot n=0\quad \text{on}\hspace{2mm}\partial\Omega
\end{array}
\right.
\end{equation*}
in the sense of distributions on $\Omega$ that admits the estimates
\begin{equation}\label{bigcee}
\|u\|_{W^{m, p}(\Omega, \mathbb{R}^{3})}\leq C_{\ast}\|F\|_{W^{m-1, p}(\Omega, \mathbb{R}^{3})}
\end{equation}
and
\begin{equation}\label{lilcee}
\|Au\|_{W^{m, p}(\Omega, \mathbb{R}^{3})}\leq c\|F\|_{W^{m-1, p}(\Omega, \mathbb{R}^{3})}
\end{equation}
for some constants $C_{\ast}=C_{\ast}(\Omega, m, p, \lambda, \mu)>0$ and $c=c(\Omega, m, p, \lambda, \mu)>0$, where the parameter $\mu$ is defined to be $\mu:=\|A\|_{C^{1, \alpha}(\bar{\Omega}, \mathbb{R}^{3\times 3})}$.
\end{prop}
By direct application of this proposition to the case when $A\equiv D^{2}P(\cdot, t)$ and the forcing term is given by $F\equiv \nabla\wedge J(\nabla P(\cdot, t)-\mathrm{id}_{\Omega})$, we deduce in the case of $m=1$ that
\begin{equation*}
\|u(\cdot, t)\|_{W^{1, p}(\Omega, \mathbb{R}^{3})}\leq C_{\ast}\left(
\| \nabla P(\cdot, t)\|_{W^{1, p}(\Omega, \mathbb{R}^{3})}+1
\right),
\end{equation*}
and, in turn, from estimate \eqref{apriorip} that
\begin{equation*}
\left(
\int_{0}^{\tau_{\ast}}\|u(\cdot, t)\|_{W^{1, p}(\Omega, \mathbb{R}^{3})}^{r}\,dt
\right)^{1/r}\leq C
\end{equation*}
for some constant $C=C(\Omega, \nabla P_{0})$ for all $1\leq r\leq \infty$.
\section{A Forward Euler Scheme in $W^{3, p}(\Omega, \mathbb{R}^{3})$}\label{yoyoyo}
Suppose that the geopotential field $\nabla P_{0}\in \mathcal{P}$ is given as an initial datum. Let $\tau>0$ be taken arbitrarily for the moment. We consider the following discrete time forward Euler scheme in $W^{3, p}(\Omega, \mathbb{R}^{3})$ on the time interval $[0, \tau]$ associated to the transport equation \eqref{AT}, namely
\begin{equation}\label{feuler}
\frac{T_{j+1}-T_{j}}{\varepsilon}=-(u_{j}\cdot\nabla)T_{j}+J(T_{j}-\mathrm{id}_{\Omega})
\end{equation}
together with
\begin{equation*}
\nabla\cdot u_{j}=0\quad \text{on}\hspace{2mm}\Omega
\end{equation*}
and the boundary condition
\begin{equation*}
u\cdot n=0\quad \text{on}\hspace{2mm}\partial\Omega,
\end{equation*}
which is supplemented with the initial datum $T_{0}:=\nabla P_{0}$. The time step length $\varepsilon>0$ of this scheme is defined to be
\begin{equation*}
\varepsilon:=\frac{\tau}{N},
\end{equation*}
for some integer $N\geq 1$. We begin with the following definition.
\begin{defn}[Classical Solution of the Forward Euler Scheme]
Suppose $\tau>0$ is given. We say that the sequences of maps 
\begin{equation*}
\{T_{j}\}_{j=0}^{N-1}\subset W^{3, p}(\Omega, \mathbb{R}^{3})\quad \text{and}\quad \{u_{j}\}_{j=0}^{N-1}\subset W^{3, p}_{\sigma}(\Omega, \mathbb{R}^{3})
\end{equation*}
constitute a {\em classical solution} of the forward Euler scheme \eqref{feuler} on $[0, \tau]$ subject to the initial datum $T_{0}:=\nabla P_{0}\in \mathcal{P}$ if and only if: 
\begin{itemize}
\item $T_{j}=\nabla P_{j}$ in $W^{3, p}(\Omega, \mathbb{R}^{3})$ for some $\lambda_{j}$-uniformly convex map $P_{j}\in C^{2}(\Omega)$ for $0\leq j \leq N-2$, and some convex map $P_{N-1}$;
\item the maps $T_{j}$ and $u_{j}$ satisfy \eqref{feuler} pointwise in the classical sense on $\Omega$ for each $0\leq j\leq N-1$.
\end{itemize}
\end{defn}
In what follows, we write $T_{j}^{N}=\nabla P_{j}^{N}$ and $u_{j}^{N}$ to denote the solution maps of the Euler scheme \eqref{feuler}, where $0\leq j\leq N-1$ and $N\geq 1$ is given and fixed. In turn, we define the approximate geopotential field $\nabla P^{N}$ and approximate Eulerian velocity field $u^{N}$ with
\begin{equation*}
\nabla P^{N}\in\mathrm{BV}([0, \tau]; W^{3, p}
(\Omega, \mathbb{R}^{3}))\quad \text{and}\quad u^{N}\in\mathrm{BV}([0, \tau]; W^{3, p}_{\sigma}(\Omega, \mathbb{R}^{3}))
\end{equation*}
pointwise on $\Omega\times [0, \tau]$ by
\begin{equation}\label{appp}
\nabla P^{N}(x, t):=\sum_{j=0}^{N-1}\mathds{1}_{I^{N}_{j}}(t)\nabla P^{N}_{j}(x)
\end{equation}
and
\begin{equation}\label{appu}
u^{N}(x, t):=\sum_{j=0}^{N-1}\mathds{1}_{I^{N}_{j}}(t)u^{N}_{j}(x),
\end{equation}
where $I^{N}_{j}\subset\mathbb{R}$ denotes the semi-open interval $I_{j}^{N}:=[j\varepsilon, (j+1)\varepsilon)$. We have not yet specified how one determines $u_{j}$, given $T_{j}$, at each time step. To illustrate a natural way by which to do this, we focus for the moment on the case of the first time step corresponding to $j=0$.
\subsection{Analysis of the First Time Step}
As mentioned above, it is evident that prescribing the initial geopotential field $P_{0}$ is not enough to determine $T_{1}^{N}$ uniquely via the expression
\begin{equation}\label{teeone}
T_{1}^{N}=\nabla P_{0}-\frac{\tau}{N} D^{2}P_{0}u_{0}^{N}+\frac{\tau}{N} J(\nabla P_{0}-\mathrm{id}_{\Omega}),
\end{equation}
since the approximate velocity field $u_{0}^{N}:\Omega\rightarrow\mathbb{R}^{3}$ is {\em not} a prescribed datum in the initial value problem. In the pursuit of a classical solution, it is evident that one must find an Eulerian velocity field $u_{0}^{N}\in W^{3, p}(\Omega, \mathbb{R}^{3})$ such that the map
\begin{equation*}
x\mapsto \nabla P_{0}(x)-\frac{\tau}{N}D^{2}P_{0}(x)u_{0}^{N}(x)+\frac{\tau}{N}J(\nabla P_{0}(x)-x)
\end{equation*}
is a conservative vector field on $\Omega$, i.e. that $T_{1}^{N}=\nabla P_{1}^{N}$ for some $P_{1}^{N}\in C^{2}(\Omega)$. There is certainly more than one way by which one might determine such a velocity field $u_{0}^{N}$. We opt, in this work, to appeal to the theory of div-curl systems to do so.

Let us consider the following auxiliary system for an unknown $w:\Omega\rightarrow \mathbb{R}^{3}$ at the first time step given by
\begin{equation}\label{divcurlzero}
\left\{
\begin{array}{l}
\nabla\wedge (D^{2}P_{0}w)=F_{0}, \vspace{2mm}\\
\nabla\cdot w=0, \vspace{2mm}\\
w\cdot n=0 \quad \text{on}\hspace{2mm}\partial\Omega,
\end{array}
\right.
\end{equation}
where $F_{0}:= \nabla \wedge J(\nabla P_{0}-\mathrm{id}_{\Omega})$. We shall term $w$ a {\em classical solution} of this div-curl system if and only if it is of class $\mathscr{C}^{1}$ and satisfies \eqref{divcurlzero} pointwise in the classical sense on $\Omega$. Now, since $\nabla P_{0}\in\mathcal{P}$, we know that the map $x\mapsto P_{0}(x)$ is $\lambda_{0}$-uniformly convex on $\Omega$ for some $\lambda_{0}>0$. In addition, as the map $F_{0}$ belongs to $W^{2, p}(\Omega, \mathbb{R}^{3})$ and satisfies the compatibility conditions
\begin{equation*}
\nabla \cdot F_{0}=0\quad \text{on}\hspace{2mm}\Omega
\end{equation*}
and
\begin{equation*}
\int_{\partial\Omega} F_{0}\cdot n\,d\mathscr{H}=0,
\end{equation*}
it is known by the recent work of Wang (\cite{wang2018solvability}, Theorem 2.1) that there exists a unique classical solution $u_{0}$ of system \eqref{divcurlzero} which satisfies the inequalities
\begin{equation}
\|u_{0}\|_{W^{3, p}(\Omega, \mathbb{R}^{3})}\leq C_{0}\|F_{0}\|_{W^{2, p}(\Omega, \mathbb{R}^{3})}
\end{equation}
and
\begin{equation}
\|D^{2}P_{0}u_{0}\|_{W^{3, p}(\Omega, \mathbb{R}^{3})}\leq c_{0}\|F_{0}\|_{W^{2, p}(\Omega, \mathbb{R}^{3})}
\end{equation}
where $C_{0}=C_{\ast}(\Omega, 3, p, \lambda, \mu_{0})>0$, $c_{0}:=c(\Omega, 3, p, \lambda_{0}, \mu_{0})$ and $\mu_{0}:=\|D^{2}P_{0}\|_{C^{1, \alpha}(\bar{\Omega}, \mathbb{R}^{3\times 3})}$. We now use this Eulerian field $u_{0}$ to `advect' $\nabla P_{0}$ forward to the next time step. Indeed, one may verify from \eqref{teeone} by taking curls across the equality that
\begin{equation*}
\nabla \wedge T_{1}^{N}=0\quad \text{on}\hspace{2mm}\Omega,
\end{equation*} 
and since $\Omega$ is open and simply connected by assumption, it follows there exists a scalar map (which we call $P_{1}^{N}$) such that $T_{1}^{N}=\nabla P_{1}^{N}$. Thus, we have established a means by which one can determine $T_{1}^{N}$ as a conservative vector field on $\Omega$. It is important to note that in the application in the results of \cite{wang2018solvability} one must have that the div-curl system of type \eqref{divcurlzero} be {\em uniformly} elliptic. This is, of course, determined by the nature of the matrix-valued map $D^{2}P_{0}$ at the first time step. As such, our scheme must also ensure that $P_{1}$ be uniformly convex on $\Omega$ in order that we can determine $u_{1}^{N}$ in an analogous manner at the next time step. 

It is at this point we turn to the establishment of a priori estimates which permit one to propagate uniform convexity of $P_{0}$ forward in time on some (possibly, short) time interval.
\subsection{A Priori Estimates for Strict Convexity of the Geopotential}
In order to construct physical classical solutions of \eqref{AT}, we work to ensure that the forward Euler scheme \eqref{feuler} maintains uniform convexity of any uniformly convex initial geopotential field on $\Omega$. As such, we seek uniform control on the quantity
\begin{equation*}
\|D^{2}P_{j}^{N}-D^{2}P_{0}\|_{L^{\infty}(\Omega, \mathbb{R}^{3\times 3})}
\end{equation*}
for the given $\lambda_{0}$-uniformly convex geopotential field $P_{0}$ in order that we may infer that
\begin{equation*}
D^{2}P_{j}^{N}(x)\xi\cdot \xi\geq \lambda_{\ast}|\xi|^{2} \quad \text{for all}\hspace{2mm}\xi\in\mathbb{R}^{3}
\end{equation*}
for some $j$- and $N$-independent constant $\lambda_{\ast}>0$. For example, one can show that if
\begin{equation*}
\|D^{2}P_{j}^{N}-D^{2}P_{0}\|_{L^{\infty}(\Omega, \mathbb{R}^{3\times 3})}\leq \frac{\lambda_{0}}{6}
\end{equation*}
for all $1\leq j\leq N-1$ and all $N\geq 2$, then through a simple application of Young's inequality, if follows that 
\begin{equation}
D^{2}P_{j}^{N}(x)\xi\cdot \xi\geq \lambda_{0}|\xi|^{2}+(D^{2}P_{j}(x)-D^{2}P_{0}(x))\xi\cdot \xi\geq \frac{\lambda_{0}}{2}|\xi|^{2},
\end{equation}
whence $P_{j}^{N}$ is a $\lambda_{0}/2$-uniformly convex function on $\Omega$. An additional and important concern in our endeavour to construct local-in-time solutions of \eqref{AT} is ensuring the control of the parameters $C_{\ast}$ and $c$ appearing in \eqref{bigcee} and \eqref{lilcee} above on $j$ and $N$, in particular their dependence on the $C^{1, \alpha}(\bar{\Omega}, \mathbb{R}^{3\times 3})$-norm of the coefficient matrix in the div-curl system at each stage of the construction.
\begin{rem}
As the reader will note, we do not try to find the time of existence $\tau_{\ast}$ which is optimal with respect to our forward Euler algorithm. We opt for a less-than-optimal result only for the sake of clean presentation of our results.
\end{rem}
\subsubsection{Some Definitions of Useful Parameters}
In the sequel, it will be helpful to make use of the following parameters. We write $\omega=\omega(\Omega, p)>0$ to denote the constant
\begin{equation*}
\omega(\Omega, p):=\|J\mathrm{id}_{\Omega}\|_{W^{3, p}(\Omega, \mathbb{R}^{3})}.
\end{equation*}
For a given initial generalised geopotential field $\nabla P_{0}\in\mathcal{P}$ which is $\lambda_{0}$-uniformly convex on $\Omega$, we also set $M_{\ast}=M_{\ast}(\nabla P_{0})>0$ to be
\begin{equation*}
M_{\ast}:=\|D^{2}P_{0}\|_{C^{1, \alpha}(\bar{\Omega}, \mathbb{R}^{3\times 3})}+\frac{\lambda_{0}}{6}.
\end{equation*}
We define the constant $c_{\ast}=c_{\ast}(\Omega, p, \nabla P_{0})$ to be
\begin{equation}
c_{\ast}(\Omega, p, \nabla P_{0}):=\max\left\{\sup_{\mu\leq M_{\ast}}c(\Omega, p, \lambda_{0}/2, \mu), \sup_{\mu\leq M_{\ast}}c(\Omega, p, \lambda_{0}, \mu)\right\},
\end{equation}
where $c(\Omega, p, \lambda, \mu)\equiv c(\Omega, 3, p, \lambda, \mu)>0$ is the constant appearing in \eqref{lilcee} above. In all that follows, we set $\tau_{\ast}=\tau_{\ast}(\Omega, p, \nabla P_{0})>0$ to be
\begin{equation}\label{tau}
\tau_{\ast}:=\frac{1}{1+2c_{\ast}}\log\left(1+\frac{\lambda_{0}}{6C_{M}(\kappa+\|\nabla P_{0}\|_{W^{3, p}(\Omega, \mathbb{R}^{3})})}\right),
\end{equation}
where $C_{M}>0$ is the constant which appears in Morrey's inequality (see, for instance, Evans \cite{evans2010partial}). Finally, we write $\kappa=\kappa(\Omega, p)>0$ to denote the constant
\begin{equation*}
\kappa(\Omega, p):=\frac{\omega+2c_{\ast}|\Omega|^{1/p}}{1+2c_{\ast}}.
\end{equation*}
We are now ready to progress to the determination of useful estimates on classical solutions of the forward Euler scheme \eqref{feuler}.
\subsubsection{Establishment of a Priori Estimates}
We now establish those estimates which are crucial in the construction of local-in-time classical solutions of \eqref{AT}. In all that follows, the H\"{o}lder exponent $\alpha\in (0, 1)$ will always be that which is determined by Morrey's inequality.
\begin{prop}
Let an initial geopotential field $\nabla P_{0}\in\mathcal{P}$ be given. Suppose that the sequences 
\begin{equation*}
\{T_{j}^{N}\}_{j=0}^{N-1}\subset W^{3, p}(\Omega, \mathbb{R}^{3})\quad  \text{and}\quad \{u_{j}^{N}\}_{j=0}^{N-1}\subset W^{3, p}_{\sigma}(\Omega, \mathbb{R}^{3})
\end{equation*}
constitute an associated classical solution of the forward Euler scheme \eqref{feuler} on the time interval $[0, \tau_{\ast}]$, where $\tau_{\ast}=\tau_{\ast}(\nabla P_{0})>0$ is given by \eqref{tau} above. It follows that
\begin{equation}\label{popone}
\|\nabla P_{j}^{N}\|_{W^{3, p}(\Omega, \mathbb{R}^{3})}\leq (\kappa+\|\nabla P_{0}\|_{W^{3, p}(\Omega, \mathbb{R}^{3})})(1+[1+2c_{\ast}]\varepsilon)^{j}-\kappa
\end{equation}
and
\begin{equation}\label{poptwo}
\|\nabla P_{j}^{N}-\nabla P_{j-1}^{N}\|_{W^{3, p}(\Omega, \mathbb{R}^{3})}\leq \varepsilon(1+2c_{\ast})(\kappa+\|\nabla P_{0}\|_{W^{3, p}(\Omega, \mathbb{R}^{3})})(1+[1+2c_{\ast}]\varepsilon)^{j-1}
\end{equation}
for each integer $j$ such that $1\leq j\leq N-1$.
\end{prop}
\begin{proof}
We achieve the proof of this proposition by means of an inductive argument. 
\subsubsection*{Case $j=1$.}
To begin, by appealing to the definition of $\nabla P_{1}^{N}$ in \eqref{feuler} above, one finds that
\begin{equation*}
\|\nabla P_{1}^{N}\|_{W^{3, p}(\Omega, \mathbb{R}^{3})}\leq (1+[1+2c_{0}]\varepsilon)\|\nabla P_{0}\|_{W^{3, p}(\Omega, \mathbb{R}^{3})}+\varepsilon(\omega+2c_{0}|\Omega|^{1/p}),
\end{equation*}
where $c_{0}>0$ is the constant appearing in \eqref{lilcee} above when $A$ is taken to be $D^{2}P_{0}$, namely
\begin{equation*}
c_{0}:=c(\Omega, 3, p, \lambda_{0}, \mu_{0}),
\end{equation*}
with $\mu_{0}:=\|D^{2}P_{0}\|_{C^{1, \alpha}(\bar{\Omega}, \mathbb{R}^{3\times 3})}$. Noting that one has the bound $c_{0}\leq c_{\ast}$, it follows that \eqref{popone} holds in the case $j=1$. A similar calculation yields that
\begin{equation*}
\|\nabla P_{1}^{N}-\nabla P_{0}\|_{W^{3, p}(\Omega, \mathbb{R}^{3})}\leq \varepsilon (1+2c_{\ast})(\kappa + \|\nabla P_{0}\|_{W^{3, p}(\Omega, \mathbb{R}^{3})}),
\end{equation*}
the straightforward details of which we leave to the reader. As a consequence of both \eqref{popone} and \eqref{poptwo} in the case $j=1$, we note that
\begin{equation*}
\begin{array}{l}
\quad \| D^{2}P_{1}^{N}\|_{C^{1, \alpha}(\bar{\Omega}, \mathbb{R}^{3\times 3})} \vspace{2mm} \\ 
\leq \| D^{2}P_{0}\|_{C^{1, \alpha}(\bar{\Omega}, \mathbb{R}^{3\times 3})}+\| D^{2}P_{1}^{N}-D^{2}P_{0}\|_{C^{1, \alpha}(\bar{\Omega}, \mathbb{R}^{3\times 3})} \vspace{2mm}\\
\leq \| D^{2}P_{0}\|_{C^{1, \alpha}(\bar{\Omega}, \mathbb{R}^{3\times 3})}+C_{M}\|\nabla P_{1}^{N}-\nabla P_{0}\|_{W^{3, p}(\Omega, \mathbb{R}^{3})}\vspace{2mm}\\
\leq \| D^{2}P_{0}\|_{C^{1, \alpha}(\bar{\Omega}, \mathbb{R}^{3\times 3})}+C_{M}\varepsilon (1+2c_{\ast})(\kappa + \|\nabla P_{0}\|_{W^{3, p}(\Omega, \mathbb{R}^{3})})\sum_{k=0}^{N-1}(1+[1+2c_{\ast}]\varepsilon)^{k},
\end{array}
\end{equation*}
from which it follows by definition of $\tau_{\ast}>0$ above that
\begin{equation*}
\| D^{2}P_{1}^{N}\|_{C^{1, \alpha}(\bar{\Omega}, \mathbb{R}^{3\times 3})}\leq \| D^{2}P_{0}\|_{C^{1, \alpha}(\bar{\Omega}, \mathbb{R}^{3\times 3})}+\frac{\lambda_{0}}{6}=M_{\ast}.
\end{equation*}
Moreover, as $\|D^{2} P_{1}^{N}-D^{2}P_{0}\|_{L^{\infty}(\Omega, \mathbb{R}^{3\times 3})}\leq \lambda_{0}/6$, it follows that $P_{1}^{N}$ is $\lambda_{0}/2$-uniformly convex on $\Omega$. As such, one may infer that
\begin{equation*}
c_{1}^{N}:=c(\Omega, 3, p, \lambda_{0}/2, \mu_{1}^{N})
\end{equation*}
with $\mu_{1}^{N}:=\|D^{2}P_{1}^{N}\|_{C^{1, \alpha}(\bar{\Omega}, \mathbb{R}^{3\times 3})}$ also admits the bound $c_{1}^{N}\leq c_{\ast}$.
\subsubsection*{Case $j=J$, $1<J<N$.}
Let us now assume it is the case that \eqref{popone} and \eqref{poptwo} hold true for some $j=J$, where $J$ is an integer chosen such that $1<J<N$. We deduce from a telescoping argument that
\begin{equation*}
\|\nabla P_{J}^{N}-\nabla P_{0}\|_{W^{3, p}(\Omega, \mathbb{R}^{3})}\leq \varepsilon(1+2c_{\ast})(\kappa+\|\nabla P_{0}\|_{W^{3, p}(\Omega, \mathbb{R}^{3})})\sum_{k=0}^{J-1}(1+[1+2c_{\ast}]\varepsilon)^{j},
\end{equation*}
from which it follows in turn that
\begin{equation*}
\|\nabla P_{J}^{N}-\nabla P_{0}\|_{W^{3, p}(\Omega, \mathbb{R}^{3})}\leq \left(\kappa+\|\nabla P_{0}\|_{W^{3, p}(\Omega, \mathbb{R}^{3})}\right)\left(
e^{(1+2c_{\ast})\tau_{\ast}}-1
\right).
\end{equation*}
Finally, an application of Morrey's inequality reveals by definition of $\tau_{\ast}>0$ that
\begin{equation*}
\| D^{2}P_{J}^{N}-D^{2}P_{0}\|_{L^{\infty}(\Omega, \mathbb{R}^{3\times 3})}\leq \frac{\lambda_{0}}{6},
\end{equation*}
whence $P_{J}^{N}$ is a $\lambda_{0}/2$-uniformly convex function on $\Omega$. Moreover, by appealing directly to \eqref{popone}, one finds that
\begin{equation*}
\|\nabla P_{J}^{N}\|_{W^{3, p}(\Omega, \mathbb{R}^{3})}\leq (\kappa+\|\nabla P_{0}\|_{W^{3, p}(\Omega, \mathbb{R}^{3})})e^{(1+2c_{\ast})\tau_{\ast}}-\kappa,
\end{equation*}
from which we deduce from yet another application of Morrey's inequality that
\begin{equation*}
\|\ D^{2}P_{J}^{N}\|_{C^{1, \alpha}(\bar{\Omega}, \mathbb{R}^{3\times 3})}\leq M_{\ast}.
\end{equation*}
Thus, under the assumption that \eqref{popone} and \eqref{poptwo} hold true, we may infer that
\begin{equation}\label{impconst}
c(\Omega, p, \lambda_{0}/2, \mu_{J}^{N})\leq c_{\ast}(\Omega, p).
\end{equation}
It is the estimate \eqref{impconst} which is crucial in passing to the final stage of our induction argument. 
\subsubsection*{Case $j=J+1$.}
By assuming that the previous case holds true, appealing to the definition of the forward Euler scheme, and making use of inequality \eqref{impconst} above, one can show that
\begin{equation}
\|\nabla P_{J+1}^{N}\|_{W^{3, p}(\Omega, \mathbb{R}^{3})}\leq (1+[1+2c_{\ast}]\varepsilon)\|\nabla P_{J}^{N}\|_{W^{3, p}(\Omega, \mathbb{R}^{3})}+\varepsilon(\omega+2c_{\ast}|\Omega|^{1/p}),
\end{equation}
from which it follows by the case $j=J$ assumption that
\begin{equation}
\|\nabla P_{J+1}^{N}\|_{W^{3, p}(\Omega, \mathbb{R}^{3})}\leq(\kappa+\|\nabla P_{0}\|_{W^{3, p}(\Omega, \mathbb{R}^{3})})(1+[1+2c_{\ast}]\varepsilon)^{J+1}-\kappa.
\end{equation}
As similar argument also can be applied to the difference $\nabla P_{J+1}^{N}-\nabla P_{J}^{N}$ in $W^{3, p}(\Omega, \mathbb{R}^{3})$. We leave the straightforward details thereof to the reader.
\end{proof}
\subsection{Construction of a Classical Solution of the Forward Euler Scheme}\label{algo}
Using the estimates established in the previous section, it is straightforward to demonstrate the existence of classical solutions of the forward Euler scheme. 
\begin{prop}[Existence of Classical Solutions of the Euler Scheme]\label{existence}
Suppose an initial datum $\nabla P_{0}\in \mathcal{P}$ and integer $N\geq 2$ are given. There exists an associated classical solution $\{\nabla P_{j}^{N}\}_{j=0}^{N-1}$ and $\{u^{N}_{j}\}_{j=0}^{N-1}$ of the Euler scheme \eqref{feuler} on $[0, \tau_{\ast}]$, where $\tau_{\ast}>0$ is given by \eqref{tau}.
\end{prop}
\begin{proof}
Omitted.
\end{proof}
With this result in place, we may now state the sense in which the approximate fields $\nabla P^{N}$ and $u^{N}$ defined by \eqref{appp} and \eqref{appu}, respectively, satisfy the active transport equation \eqref{AT}. 
\begin{prop}\label{preAT}
Suppose $\nabla P_{0}\in \mathcal{P}$, and let $\tau_{\ast}>0$ denote the existence time given by \eqref{tau} above. Let $N\geq 1$ be given and fixed. The associated approximate geopotential field $\nabla P^{N}:\Omega\times [0, \tau_{\ast}]\rightarrow\mathbb{R}^{3}$ and approximate Eulerian velocity field $u^{N}:\Omega\times [0, \tau_{\ast}]\rightarrow\mathbb{R}^{3}$ satisfy the integral equality
\begin{align}\label{preat}
\int_{0}^{\tau_{\ast}}\int_{\Omega}
\nabla P^{N}(x, t)\cdot \left(\frac{\psi(x, t-\varepsilon)-\psi(x, t)}{-\varepsilon}\right)+(u^{N}(x, t)\cdot\nabla)\psi(x, t)\cdot\nabla P^{N}(x, t)\,dxdt \vspace{2mm} \notag \\
=-\int_{0}^{\tau_{\ast}}\int_{\Omega}J(\nabla P^{N}(x, t)-x)\cdot\psi(x, t)\,dxdt
\end{align}
for all $\psi\in C^{\infty}_{c}(\Omega\times (0, \tau_{\ast}), \mathbb{R}^{3})$ with the property $\mathrm{supp}\,\psi\subset\subset \Omega\times (\varepsilon, \tau_{\ast}-\varepsilon)$.
\end{prop}
\begin{proof}
This follows immediately from the definition of the scheme \eqref{feuler}.
\end{proof}
To conclude our proof of theorem \ref{mainres}, one task remains to be completed. Indeed, we must establish compactness of the sequences of approximants $\{\nabla P^{N}\}_{N=1}^{\infty}$ and $\{u^{N}\}_{N=1}^{\infty}$ in suitable topologies, so that we may pass to the limit in \eqref{preat} above as $N\rightarrow \infty$ in order to recover a local-in-time physical classical solution of \eqref{AT}. We tackle this task in the following section.
\subsection{Uniform Bounds on Approximants $\nabla P^{N}$ and $u^{N}$}
As the approximants $\nabla P^{N}$ lie in function spaces whose members are of bounded variation in time, it is natural to seek to establish compactness of $\{\nabla P^{N}\}_{N=1}^{\infty}$ in $\mathrm{BV}([0, \tau_{\ast}]; W^{3, p}(\Omega, \mathbb{R}^{3}))$. On the other hand, by way of the estimates in \cite{wang2018solvability}, it will be convenient to establish compactness of $\{u^{N}\}_{N=1}^{\infty}$ in the space $L^{r}(0, \tau_{\ast}; W^{3, p}(\Omega, \mathbb{R}^{3}))$ instead, for some $r>1$. For the convenience of the reader, we state the following version of Helly's Selection Principle for maps on compact intervals of the real line with range in a reflexive Banach space. The result we quote is taken from Barbu and Precupanu (\cite{barbu2012convexity}, theorem 1.126) in the special case when the Banach space is taken to be $W^{3, p}(\Omega, \mathbb{R}^{3})$.
\begin{thm}[Helly's Selection Principle]\label{helly}
Suppose $1<p<\infty$. Let the sequence $\{S^{N}\}_{N=1}^{\infty}\subset\mathrm{BV}([0, \tau_{\ast}]; W^{3, p}(\Omega, \mathbb{R}^{3}))$ be such that
\begin{equation*}
\|S^{N}(\cdot, t)\|_{W^{3, p}(\Omega, \mathbb{R}^{3})}\leq C\quad \text{for all}\hspace{2mm}t\in [0, \tau^{\ast}]
\end{equation*}
and
\begin{equation*}
\mathrm{Var}(S^{N}; [0, \tau_{\ast}])\leq C
\end{equation*}
for some constant $C>0$ independent of $N\geq 1$. It follows that there exists a (relabelled) subsequence of $\{S^{N}\}_{N=1}^{\infty}$ and a map $S\in\mathrm{BV}([0, \tau^{\ast}]; W^{3, p}(\Omega, \mathbb{R}^{3}))$ such that
\begin{equation*}
S^{N}(\cdot, t)\rightharpoonup S(\cdot, t)\quad \text{in}\hspace{2mm}W^{3, p}(\Omega, \mathbb{R}^{3})
\end{equation*}
for all $t\in [0, \tau^{\ast}]$, and
\begin{equation*}
\int_{0}^{\tau_{\ast}}\psi\,dS^{N}\rightharpoonup\int_{0}^{\tau_{\ast}}\psi\,dS\quad\text{in}\hspace{2mm}W^{3, p}(\Omega, \mathbb{R}^{3})
\end{equation*}
for all $\psi \in C^{0}([0, \tau^{\ast}])$.
\end{thm}
With this compactness result in hand, we now consider the extraction of convergent subsequences of both $\{\nabla P^{N}\}_{N=1}^{\infty}$ and $\{u^{N}\}_{N=1}^{\infty}$ in order that we may pass to a local-in-time classical solution of \eqref{AT}.
\subsubsection{Bounds on the Geopotential Fields $\{\nabla P^{N}\}_{N=1}^{\infty}$}
We begin by noting directly from \eqref{popone} above that
\begin{equation*}
\|\nabla P^{N}(\cdot, t)\|_{W^{3, p}(\Omega, \mathbb{R}^{3})}\leq \left(\kappa + \|\nabla P_{0}\|_{W^{3, p}(\Omega, \mathbb{R}^{3})}\right)e^{(1+2c_{\ast})\tau_{\ast}}
\end{equation*}
for any $0\leq t\leq \tau_{\ast}$ and $N\geq 2$. It also follows from \eqref{poptwo} that
\begin{equation*}
\mathrm{Var}(\nabla P^{N}; [0, \tau_{\ast}])\leq \left(
\kappa +\|\nabla P_{0}\|_{W^{3, p}(\Omega, \mathbb{R}^{3})}
\right)e^{(1+2c_{\ast})\tau_{\ast}}
\end{equation*}
for all $N\geq 2$. We may infer immediately from Theorem \ref{helly} above that the sequence $\{\nabla P^{N}\}_{N=1}^{\infty}$ constitutes a precompact set in $\mathrm{BV}([0, \tau_{\ast}]; W^{3, p}(\Omega, \mathbb{R}^{3}))$, whence we deduce the existence of a limit map $T\in \mathrm{BV}([0, \tau_{\ast}]; W^{3, p}(\Omega, \mathbb{R}^{3}))$ with the property that
\begin{equation*}
\nabla P^{N}(\cdot, t)\rightharpoonup T(\cdot, t)\quad \text{in}\hspace{2mm} W^{3, p}(\Omega, \mathbb{R}^{3})
\end{equation*}
for every $t\in [0, \tau_{\ast}]$. It is important to know that the limit map $T$ is a time-dependent conservative vector field on $\Omega$. We may assume, without loss of generality, that
\begin{equation*}
\int_{\Omega}P^{N}_{j}(y)\,dy=0
\end{equation*}
for all $0\leq j\leq N-1$ and $N\geq 1$. In turn, by the Poincar\'{e}-Wirtinger inequality, it follows that
\begin{equation*}
\|P^{N}(\cdot, t)\|_{L^{p}(\Omega)}\leq C\|\nabla P^{N}(\cdot, t)\|_{W^{3, p}(\Omega, \mathbb{R}^{3})}
\end{equation*}
for some constant $C=C(\Omega, p)>0$ and so, by the extraction of yet another (relabelled) subsequence $\{\nabla P^{N}\}_{N=1}^{\infty}$, it follows by the Rellich-Kondrashov Theorem that
\begin{equation*}
\nabla P^{N}(\cdot, t)\rightarrow \nabla P(\cdot, t)\quad \text{in}\hspace{2mm} W^{1, p}(\Omega, \mathbb{R}^{3})
\end{equation*}
as $N\rightarrow\infty$ for some scalar-valued function $P(\cdot, t)$. Moreover, as we are guaranteed the pointwise convergence result
\begin{equation*}
D^{2}P^{N}(x, t)\rightarrow D^{2}P(x, t)\quad \text{in}\hspace{2mm}\mathbb{R}^{3\times 3}
\end{equation*}
as $N\rightarrow\infty$ for all $t\in [0, \tau_{\ast}]$ and $x\in\Omega$, it follows that $P(\cdot, t)$ is $\lambda_{0}/2$-uniformly convex on $\Omega$ for $t\in [0, \tau_{\ast}]$. Finally, by the Dominated Convergence Theorem, it follows that
\begin{equation*}
\lim_{N\rightarrow\infty}\left(\int_{0}^{\tau_{\ast}}\|\nabla P^{N}(\cdot, t)-\nabla P(\cdot, t)\|_{L^{p}(\Omega, \mathbb{R}^{3})}^{q}\right)^{1/q}=0
\end{equation*}
for any $1\leq q<\infty$. This strong convergence of the gradient geopotential fields will allow us to pass to the limit as $N\rightarrow\infty$ in \eqref{preat} equipped additionally with only the weak compactness of the set of Eulerian velocities $\{u^{N}\}_{N=1}^{\infty}$.
\subsubsection{Bounds on the Velocity Field}
It is at this point we may appeal directly to the estimates on classical solutions of the $\mathrm{div}$-$\mathrm{curl}$ system quoted in section \ref{apriori} above. From our work above it follows that
\begin{equation}\label{yoo}
\|u^{N}(\cdot, t)\|_{W^{3, p}(\Omega, \mathbb{R}^{3})}\leq 2c_{\ast}(\| \nabla P^{N}(\cdot, t)\|_{W^{3, p}(\Omega)}+|\Omega|^{1/p}),
\end{equation}
and so by estimate \eqref{popone} it follows that $\{u^{N}\}_{N=1}^{\infty}$ is uniformly bounded in $L^{r}(0, \tau_{\ast}; W^{3, p}(\Omega, \mathbb{R}^{3}))$ for any $1<r<\infty$. We infer the existence of a limit map $u\in L^{r}(0, \tau_{\ast}; W^{3, p}(\Omega, \mathbb{R}^{3}))$ for which one has that
\begin{equation*}
u^{N}\rightharpoonup u\quad \text{in}\hspace{2mm}L^{r}(0, \tau_{\ast}; W^{3, p}_{\sigma}(\Omega, \mathbb{R}^{3}))
\end{equation*}
as $N\rightarrow\infty$. Weak convergence of $u^{N}$ to its limit will suffice for the purposes of constructing a local-in-time classical solution of \eqref{AT}.
\subsection{Passing to the Limit as $N\rightarrow\infty$}
Let us suppose $\nabla P_{0}\in \mathcal{P}$ is given. For each $N\geq 1$, by proposition \ref{existence}, there exists a classical solution $\{\nabla P^{N}_{j}\}_{N=0}^{N-1}$ and $\{u^{N}_{j}\}_{j=0}^{N-1}$ of the forward Euler scheme on $[0, \tau_{\ast}]$, where $\tau_{\ast}$ is given by \eqref{tau} above. By proposition \ref{preAT} above, for any given $\psi\in C^{\infty}_{c}(\Omega\times (0, \tau_{\ast}), \mathbb{R}^{3})$, there exists an integer $N_{\ast}=N_{\ast}(\psi)\geq 1$ such that 
\begin{align}\label{prelim}
-\int_{0}^{\tau_{\ast}}\int_{\Omega}\nabla P^{N}(x, t)\cdot\left(
\frac{\psi(x, t-\tau/N)-\psi(x, t)}{-\tau_{\ast}/N}
\right)+(u^{N}(x, t)\cdot\nabla)\psi(x, t)\cdot\nabla P^{N}(x, t)\,dxdt\vspace{2mm} \notag \\
=\int_{0}^{\tau_{\ast}}\int_{\Omega}J(\nabla P^{N}(x, t)-x)\cdot\psi(x, t)\,dxdt
\end{align}
holds true for all $N\geq N_{\ast}$. Using the previously-established fact that
\begin{equation*}
\nabla P^{N}\rightarrow \nabla P\quad \text{in}\hspace{2mm} L^{q}(0, \tau_{\ast}; W^{3, p}(\Omega, \mathbb{R}^{3}))
\end{equation*}
for any $1\leq q<\infty$, together with
\begin{equation*}
u^{N}\rightharpoonup u\quad \text{and}\hspace{2mm} L^{r}(0, \tau_{\ast}; W^{3, p}_{\sigma}(\Omega, \mathbb{R}^{3}))
\end{equation*}
for any $1<r<\infty$, it follows by passing to the limit in \eqref{prelim} that
\begin{equation}\label{weaksauce}
\int_{0}^{\tau_{\ast}}\int_{\Omega}\nabla P(x, t)\cdot\partial_{t}\psi(x, t)+(u(x, t)\cdot\nabla)\psi(x, t)\cdot\nabla P(x, t)+J(\nabla P(x, t)-x)\cdot\psi(x, t)\,dxdt=0
\end{equation}
for all $\psi\in C^{\infty}_{c}(\Omega\times (0, \tau_{\ast}), \mathbb{R}^{3})$, whence $\nabla P$ and $u$ constitute a local-in-time weak solution of \eqref{AT}. To obtain the proof of theorem \ref{mainres}, it remains only to show that $\nabla P$ and $u$ constitute a local-in-time classical solution of \eqref{AT}.
\subsection{Proof of Theorem \ref{mainres}}
We begin by noting that identity \eqref{weaksauce} implies that the distributional time derivative $\dot{\nabla P}$ is given explicitly by
\begin{equation}\label{ppp}
\dot{\nabla P}=-D^{2}Pu+J(\nabla P-\mathrm{id}_{\Omega})\quad \text{in}\hspace{2mm}L^{q}(0, \tau_{\ast}; W^{3, p}(\Omega, \mathbb{R}^{3})),
\end{equation}
from which we readily deduce that $\nabla P\in C^{0}([0, \tau_{\ast}]; W^{2, 2}(\Omega, \mathbb{R}^{3}))$. We aim to show that the distributional time derivative of $\nabla P$ is in fact classical. To do so, we shall show that $u$ admits some continuity properties, in particular $u\in C^{0}([0, \tau_{\ast}], W^{2, 2}_{\sigma}(\Omega, \mathbb{R}^{3}))$. Indeed, in this direction, if we choose the test function $\psi$ to be of the shape $\nabla\wedge\varphi$ for some $\varphi\in C^{\infty}_{c}(\Omega\times (0, \tau_{\ast}), \mathbb{R}^{3})$, it follows from an application of integration by parts that
\begin{equation*}
\int_{0}^{\tau_{\ast}}\int_{\Omega}\nabla\wedge (D^{2}P(x, t)u(x, t)-\nabla\wedge(\nabla P(x, t)-x))\cdot\varphi(x, t)\,dxdt=0
\end{equation*}
for all $\varphi\in C^{\infty}_{c}(\Omega\times (0, \tau_{\ast}), \mathbb{R}^{3})$. We infer that for Lebesgue a.e. $t\in [0, \tau_{\ast}]$, it holds that (any member of the equivalence class associated to) the Eulerian velocity field $u(\cdot, t)\in W^{3, p}_{\sigma}(\Omega, \mathbb{R}^{3})$ solves
\begin{equation*}
\left\{
\begin{array}{l}
\nabla\wedge(D^{2}P(\cdot, t)u(\cdot, t))=\nabla\wedge(\nabla P(\cdot, t)-\mathrm{id}_{\Omega}), \vspace{2mm}\\
\nabla \cdot u(\cdot, t)=0, \vspace{2mm}\\
u(\cdot, t)\cdot n=0\quad \text{on}\hspace{2mm}\partial\Omega
\end{array}
\right.
\end{equation*}
pointwise in the classical sense on $\Omega$. We note the right-hand side of the curl equation can be considered as an element of $C^{0}([0, \tau_{\ast}]; W^{1, 2}(\Omega, \mathbb{R}^{3}))$. As the map $t\mapsto \nabla \wedge(\nabla P(\cdot, t)-\mathrm{id}_{\Omega})$ is defined everywhere on $[0, \tau_{\ast}]$, by redefining the (representative) $u$ on a set of measure zero to be the unique classical solution of the above div-curl system, we may appeal to the following simple lemma.
\begin{lem}[Stability]
Suppose $\{A_{k}\}_{k=1}^{\infty}\subset W^{2, p}(\Omega, \mathbb{R}^{3\times 3})$ and $\{F_{k}\}_{k=1}^{\infty}\subset W^{2, p}(\Omega, \mathbb{R}^{3})$ are convergent sequences, i.e.
\begin{equation*}
A_{k}\rightarrow A\in W^{2, p}(\Omega, \mathbb{R}^{3\times 3}) \quad \text{and}\quad F_{k}\rightarrow F\in W^{2, p}(\Omega, \mathbb{R}^{3})
\end{equation*}
in their respective topologies as $k\rightarrow\infty$. Let $u_{k}\in W^{3, p}_{\sigma}(\Omega, \mathbb{R}^{3})$ denote the unique classical solution of the div-curl system
\begin{equation*}
\left\{
\begin{array}{l}
\nabla\wedge(A_{k}u_{k})=F_{k},\vspace{2mm}\\
\nabla\cdot u_{k}=0, \vspace{2mm}\\
u_{k}\cdot n=0\quad \text{on}\hspace{2mm}\partial\Omega.
\end{array}
\right.
\end{equation*}
If follows that $u_{k}\rightharpoonup u$ in $W^{3, p}_{\sigma}(\Omega, \mathbb{R}^{3})$ as $k\rightarrow\infty$, where $u$ is the unique classical solution of the div-curl system
\begin{equation*}
\left\{
\begin{array}{l}
\nabla\wedge(Au)=F,\vspace{2mm}\\
\nabla\cdot u=0, \vspace{2mm}\\
u\cdot n=0\quad \text{on}\hspace{2mm}\partial\Omega.
\end{array}
\right.
\end{equation*}
\end{lem}
As a consequence of the above lemma, we infer that $u\in C^{0}([0, \tau_{\ast}], W^{2, 2}_{\sigma}(\Omega, \mathbb{R}^{3}))$. Finally, we infer from identity \eqref{ppp} that $\dot{\nabla P}\in C^{0}([0, \tau_{\ast}], W^{2, 2}(\Omega, \mathbb{R}^{3}))$, and in turn that $\nabla P$ and $u$ constitute a local-in-time classical solution of \eqref{AT}. This ends the proof of theorem \ref{mainres}.
\section{Brief Remarks on the Variable Rotation Model}
Let us now make some brief comments as to how one can extend what has been achieved in this article to the construction of local-in-time physical classical solutions of the incompressible semi-geostrophic system with variable Coriolis force. Suppose $f:\Omega\rightarrow\mathbb{R}$ is a given function of class $\mathscr{C}^{1}$ with the property that
\begin{equation}\label{slowspace}
\|f\|_{L^{\infty}(\Omega)}\ll 1\quad \text{and}\quad \|\nabla f\|_{L^{\infty}(\Omega, \mathbb{R}^{3})}\ll 1,
\end{equation}
which we subsequently interpret at a spatially-varying Coriolis force. The associated semi-geostrophic system takes the following form (c.f. Hoskins \cite{hoskins1975geostrophic}, page 236):
\begin{equation}\label{vary}
\left(\frac{\partial}{\partial t}+u\cdot\nabla\right)\left(K_{f}\nabla \phi\right)-fJ^{2}u=J\nabla \phi,
\end{equation}
where $\phi:\Omega\times [0, \infty)\rightarrow\mathbb{R}$ is the unknown geopotential, and $K_{f}:\Omega\rightarrow\mathbb{R}^{3\times 3}$ is given by
\begin{equation*}
K_{f}(x):=\left(
\begin{array}{ccc}
\frac{1}{f(x)} & 0 & 0 \\
0 & \frac{1}{f(x)} & 0 \\
0 & 0 & 1
\end{array}
\right) \quad \text{for}\hspace{2mm}x\in\Omega.
\end{equation*}
With section \ref{yoyoyo} above in mind, we posit that a suitable forward Euler scheme associated to this system is given by
\begin{equation}
\nabla \phi_{j+1}=\nabla \phi_{j}-\varepsilon \left(D^{2}\phi_{j}-K_{f}^{-1}\frac{\nabla \phi_{j}\otimes \nabla f}{f^{2}}\right)u_{j}+\varepsilon K_{f}^{-1}J\nabla \phi_{j},
\end{equation}
for $j=0, ..., N-1$, which is subject to the initial condition $\nabla \phi_{0}\in\mathcal{P}$. Owing to the fact that the `physical' requirement \eqref{slowspace} need hold true, it follows that at each time step, the matrix-valued term $f^{-2}K_{f}^{-1}(\nabla \phi_{j}\otimes \nabla f)$ can be viewed as a small perturbation of the Hessian of the uniformly convex map $\phi_{j}$. As such, condition \eqref{slowspace} allows one to propagate uniform convexity of the approximate geopotential by applying the results of Wang \cite{wang2018solvability} to the uniformly elliptic system
\begin{equation*}
\left\{
\begin{array}{l}
\nabla\wedge\left(\left(D^{2}\phi_{j}-K_{f}^{-1}\frac{\nabla \phi_{j}\otimes \nabla f}{f^{2}}\right)u_{j}\right)=\nabla\wedge K_{f}^{-1}J\nabla \phi_{j},\vspace{2mm}\\
\nabla\cdot u_{j}=0, \vspace{2mm}\\
u_{j}\cdot n=0,
\end{array}
\right.
\end{equation*}
for $j=0, ..., N-1$. We leave the details of the construction of local-in-time classical solutions of \eqref{vary} to the reader.
\section{Closing Remarks}\label{final}
In this article, we have tackled and solved the open problem on the existence of local-in-time physical classical solutions of the incompressible semi-geostrophic system \eqref{SG} on bounded simply-connected domains. Moreover, our solutions $\nabla P$ admit the important physical property that the geostrophic measure $\nabla P(\cdot, t)\#\mathscr{L}_{\Omega}$ is of bounded support in $\mathbb{R}^{3}$, which is a distinct novelty in our work. We have, however, made no effort here to demonstrate uniqueness of these solutions, but rather refer the reader to section 9 of Cheng, Cullen and Feldman \cite{cheng2017solvability} and the work of Feldman and Tudorascu \cite{feldman2017semi} for related results thereon. In order to extend these efforts to the construction of global-in-time weak solutions of \eqref{SG}, one must refine one's study of the propagation of uniform convexity of the geopotential $P(\cdot, t)$ in Eulerian co-ordinates. We consider this in future work.  
\subsection*{Acknowledgements}
The author would like to thank Mike Cullen, Beatrice Pelloni, Endre S\"{u}li and Jacques Vanneste for stimulating discussions related to the material in this article. I would also like to thank Daniel Coutand and Steve Shkoller for clarification on certain properties of div-curl systems. I would also like to thank an anonymous referee whose comments led to improvements on an earlier version of this article.

\appendix

\bibliography{bib}

\end{document}